\documentclass[11pt, english]{amsart}
\usepackage{babel,amssymb,amsmath,latexsym,amsthm,amscd,eucal,enumerate,verbatim,
setspace}  
\usepackage[dvips]{graphicx, color}
\usepackage[all]{xy}
\usepackage[left=2.5cm,top=2.5cm, bottom=2.5cm,right=2.5cm]{geometry}

\parskip=\smallskipamount

\newtheorem{theorem}{Theorem}[section]
\newtheorem{lemma}[theorem]{Lemma}
\newtheorem{corollary}[theorem]{Corollary}
\newtheorem{proposition}[theorem]{Proposition}

\theoremstyle{definition}
\newtheorem{definition}[theorem]{Definition}
\newtheorem{example}[theorem]{Example}

\newtheorem{remark}[theorem]{Remark}

\numberwithin{equation}{section}

\newcommand{\x}{\bf{x}\rm}

\newcommand{\cod}{\mbox{cod}\,}

\newcommand{\rr}{\mathbb{R}}

\newcommand{\re}[1]{\rr^{#1}}
\newcommand{\avb}[3]{#1 :\re{#2}\to \re{#3}}
\newcommand{\map}[3]{#1 : #2 \to #3}
\newcommand{\rom}[4]{#1^{#2}(#3,#4)}
\newcommand{\der}[2]{\frac{\partial #1}{\partial #2}}

\newcommand{\en}[1]{\left\| #1 \right\|}
\newcommand{\enorm}[1]{\left\| #1 \right\|}

\newcommand{\as}{\mbox{Ast}}

\newcommand{\tse}[1]{\tilde S_{\epsilon}(#1)}
\newcommand{\sep}{S_{\epsilon}}

\newcommand{\sigfn}{\Sigma(\hat f) \setminus \{0\}}

\begin{document}
\bibliographystyle{plain}
\title{Topological equivalence of finitely determined real analytic plane-to-plane
map germs}
\author{Olav Skutlaberg}
\address{Olav Skutlaberg: Matematisk Institutt, Universitetet i Oslo,
Postboks 1053 Blindern, 0316 Oslo, Norway}
\email{oskutlab@math.uio.no}
%
%
\subjclass[2000]{}
\date{\today}

\subjclass[2000]{}
\date{\today}
\maketitle
\begin{abstract}
Generic smooth map germs $(\re 2,0)\to (\re 2,0)$ are topologically
equivalent to cones of mappings $S^1\to S^1$. We carry out a complete
topological classification of smooth stable mappings of the circle and
show how this classification leads, via the result mentioned above, to
a topological classification of finitely determined  real analytic map germs $(\re 2,0)\to (\re 2,0)$.
\end{abstract}
\section{Introduction}
Let $f$ and $g$ be smooth mappings between smooth manifolds $N$ and $P$ of dimensions $n$ and $p$, respectively. Let $0\le r\le \infty$. We say that $f$ and $g$ are $\mathcal A_r$-equivalent if there is a commutative diagram 
\[
\begin{CD}
N@>f>>P\\
@VhVV	@VVkV\\
N@>g>>P
\end{CD}
\]
where $h$ and $k$ are $C^r$ diffeomorphisms. Similarly, if $f$ and $g$ are smooth map germs $(N,p)\to (P,q)$, then we say that $f$ and $g$ are $\mathcal A_r$-equivalent if there is a commutative diagram 
\[
\begin{CD}
(N,p)@>f>>(P,q)\\
@VhVV	@VVkV\\
(N,p)@>g>>(P,q)
\end{CD}
\]
where $h$ and $k$ are germs of $C^r$ diffeomorphisms. $\mathcal
A_0$-equivalence is usually referred to as topological
equivalence. Let $\rom C{\infty}NP$ be the set of proper smooth mappings
$N\to P$, and let $\rom C{\infty}np$ (resp. $\mathcal O(n,p)$) be the set
of smooth (resp. real analytic) map germs $(\re n,0)\to (\re p,0)$.

A subset $\Sigma\subset \rom C{\infty}np$ (resp. $\rom {\mathcal O}{}np$) is
\emph{proalgebraic} if 
\[
\Sigma = \bigcap_{r\geq 1} (j^r)^{-1}(\Sigma_r)
\] 
where each $\Sigma_r\subset J^r(n,p)$ is an algebraic subvariety. A
proalgebraic set $\Sigma$ is of \emph{infinite codimension} if 
\[
\lim_{r\to \infty} \cod \Sigma_r = \infty.
\]
A property of smooth (real analytic resp.) germs is said to hold
\emph{in general} if the set of germs not having the property is
contained in a proalgebraic set of infinite codimension. 

 By the cone of a smooth map $f:S^{n-1}\to S^{p-1}$, we mean the map $F: S^{n-1}\times [0,1)/S^{n-1}\times \{0\} \to S^{p-1}\times [0,1)/S^{p-1}\times \{0\}$ given by 
\[
F([(p,t)])=[(f(p),t)].
\] 

Consider the space $\rom C{\infty}np$ when $n\le p$, $n\ne 4,5$ and
$(n,p)$ is in the 'nice range'. The 'nice range' consists of the pairs of
dimensions of $N$ and $P$ such that the set of proper smooth stable
mappings $N\to P$ are dense in the set of proper smooth mappings $N\to P$. It is shown in \cite{Fuk1} that for
germs in $\rom C{\infty}np$, the property of having a realization
which is topologically equivalent to the
cone of a smooth stable mapping $S^{n-1}\to S^{p-1}$ via
homeomormphisms which are diffeomorphisms outside the origin holds in
general. We say that map-germs with this property are \emph{generic}.
Thus, for $n,p$ in this range, the classification of generic map germs $(\re n,0)\to(\re p,0)$ with respect to topological equivalence is contained in the classification of the smooth stable mappings $S^{n-1}\to S^{p-1}$ in the sense that the $\mathcal A_0$-equivalence class in $C^{\infty}(n,p)$ of a generic map germ corresponds to an $\mathcal A_{\infty}$-equivalence class in $C^{\infty}(S^{n-1},S^{p-1})$.

In this paper we carry out this classification in the real analytic
case for $n=p=2$. In
Section \ref{S:class} we classify the smooth stable mappings $S^1\to
S^1$ and show how to generate complete lists of the $\mathcal
A_{\infty}$-equivalence classes of such mappings. In the case of
1-dimensional spheres, the classification is essentially a combinatorical problem. In Section
\ref{S:class2} we classify finitely determined real analytic map germs $(\re
2,0)\to (\re 2,0)$ using the above strategy. Our method solves
the so-called 'recognition problem': Given two finitely determined real analytic map germs $(\re
2,0)\to (\re 2,0)$, are they $\mathcal A_0$-equivalent?


\section{Classification of smooth stable mappings $S^1 \to S^1$}\label{S:class}
Let $\map f{S^1}{S^1}$ be a smooth
stable mapping. Then $f$ has only Morse singularities, $\Sigma(f)$ is
finite and $f$ has no singular double points. If we traverse the
circle in the source counter-clockwise and registar the singular points of $f$ and
the pre-images of singular values of $f$ we encounter, then the
topological type of $f$ is in some sense determined by the pattern
arising. In the following we give the last sentence precise content.


\subsection{Definition of $\as(f)$}
Let $\map P{[0,2\pi)}{S^1}$ be the parametrization given by
$P(t) = (\cos t, \sin t)$.
If $f$ has no singular points, then we define $\as(f)=(p,p,\ldots,p)$
where $p$ is repeated $\#f^{-1}(1)$ times. 
Assume $f$ has singular points $s_i(f)$, $i=1,\ldots,n(f)$ where 
$k<l \Rightarrow P^{-1}(s_k(f))<P^{-1}(s_l(f))$.
Let
$\sigma_i(f) = f(s_i(f))$ and let $f^{-1}(\sigma_i(f))\setminus\{s_i(f)\} = \{p_{ij}(f)\}_{j=1}^{m_i}$
where  $k<l \Rightarrow P^{-1}(p_{ik}(f))<P^{-1}(p_{il}(f))$. Let 
\[
A(f)=\{a_1,a_2,\ldots,a_N\}=P^{-1}\left(\bigcup_{i=1}^n(\{s_i(f)\}\cup
  \{p_{ij}(f)\}_{j=1}^{m_i})\right)
\]
where $i<j\Rightarrow a_i<a_j$ and $N=N(f)=n(f)+\sum_{i=1}^{n(f)}
m_i$. Let $\Delta(f)=f(\Sigma(f))=\{\sigma_1(f),\ldots,
\sigma_n(f)\}$ and define
\[
B(f)=\{b_1,b_2,\ldots, b_{n(f)}\} = P^{-1}(\Delta(f))
\]
where $i<j\Rightarrow b_i<b_j$.

Next, let
\[S=\{s,p\}, \quad S^* = \bigcup_{i=1}^{\infty}\{s_i,p_i\}\]
and define maps $\map T{A(f)}S$ and $\map {T^*}{A(f)}{S^*}$ given by
\[
T(x)=\begin{cases}
s, &\text{if $P(x)=s_i(f)$;}\\
p, &\text{if $P(x) = p_{ij}(f)$}
\end{cases}
, \quad
T^*(x)= \begin{cases}
s_i, &\text{if $P(x)=s_i(f)$;}\\
p_i, &\text{if $P(x) = p_{ij}(f)$}
\end{cases}.
\]
Now, define the \emph{associated tuples} of $f$ to be the ordered $N(f)$-tuples
\[
\as(f) = \big(T(a_1),T(a_2),\ldots,T(a_{N(f)})\big)
\]
and
\[
\as^*(f)= \big(T^*(a_1),T^*(a_2),\ldots,T^*(a_{N(f)})\big)
\]

\begin{figure}
\begin{center}
\resizebox{13cm}{!}{\input{art4-1.pstex_t}}
\caption{Visualization of a map $f:S^1\to S^1$ with
  $\as(f)=(p,s,s,p,p,s,s,p)$. The curve $c:[0,2\pi)\to \re 2$ is such
  that $c(t)/\en{c(t)}=f(P(t))$. }
\label{F:1}
\end{center}
\end{figure}

\begin{remark}\label{R:tup}
Given $\as^* (f)$ one can obtain $\as(f)$ by just forgetting the
indices of the $s$ and $p$ in $\as^*(f)$. Conversely, given $\as(f)$,
it is easy to find the right indices for the $s$ in $\as^*(f)$ and
then we can find the indices of $p$ in $\as^*(f)$ as well, using the
fact that at a singular point, $f$ changes the behaviour of being
orientation preserving or orientation reversing. This enables us to
find the correct indices of the $p$.
\end{remark}

\subsection{Legal permutations}\label{SS:gr}
Let $S_k$ be the group of permutations of $\mathbb Z/k\mathbb Z$. Some permutations are of particular interest when
trying to classify stable maps under $\mathcal A_0$-equivalence. We start
with some definitions.

\begin{definition}
An element $\sigma\in S_k$ is a \emph{switch} if there is some $a\in
\mathbb Z$ such that 
\[
\sigma([x])=[x+a].
\]
Let $Sw_k$ be the set of switches in $S_k$.
\end{definition}
\begin{definition}
The permutation $r\in S_k$ given by $r([x])=[-x]$ is called the
\emph{reversation}. Let $R_k = \{ \mbox{id}, r\}$.
\end{definition}

\begin{definition}
The subgroup $L_k=\{\sigma\circ \tau \mid \sigma\in Sw_k, \tau \in R_k\}$
 of $S_k$ called the group of \emph{legal permutations}.
\end{definition} 

Let $X$ be a set. For every $k$, let $e_k:\{1,2,\ldots,k\}\to \mathbb
Z/k\mathbb Z$ be the bijection $x\mapsto [x]$. We introduce an
equivalence relation $E_k$ on
$X^k$ by
the rule $(t_1,t_2,\ldots t_k)\sim (t'_1,t'_2,\ldots, t'_k)$ if there is a permutation $\rho \in L_k$ such that
$t_i=t'_{e_k^{-1}(\rho(e_k(i)))}$ for all $i=1,\ldots, k$. Denote the
$E_k$-equivalence class of $t\in X^k$ by $[t]_E$. For $\rho\in S_k$
and $t\in X^k$, let $\rho\cdot t\in X^k$ be defined by $(\rho
\cdot t)_i=t_{e_k^{-1}(\rho(e_k(i)))}$ for $i=1,\ldots k$. For
simplicity we write $\rho(i)$ for $e_k^{-1}(\rho(e_k(i)))$.


\subsection{The main theorem of the classification}
The aim of this section is to prove the following theorem. 

\begin{theorem}\label{T:MT1}
Let $f,g\in \rom C{\infty}{S^1}{S^1}$ be $C^{\infty}$-stable. Then  
\[
f\sim_{\mathcal A_{\infty}}g \Leftrightarrow N(t)=N(g) \mbox{\, and \,}[\as (f)]_E=[\as(g)]_E.
\]
\end{theorem}
\begin{proof}
We prove the theorem when $\Sigma(f),\Sigma(g)\ne \emptyset$. The same
technique applies when $\Sigma(f=)\Sigma(g)=\emptyset$.
The theorem is proved in three steps:

\emph{Step 1} is to prove that $\as(f)=\as(g) \Rightarrow f\sim_{\mathcal A_{\infty}}g$. Suppose   $\as(f)=\as(g)$. After composition with diffeomorphisms in source, we may assume that $A(f)=A(g)$ and that 1 is a regular point of $f$, and hence also of $g$. A priori, it may happen that $f$ is orientation preserving on $P([a_1(f),a_2(f)])$ while $g$ is not, but after composition with a diffeomorphism in target, we may assume that $f$ and $g$ are orientation preserving on the same subset of source, and that $\sigma_i(f)=\sigma_i(g)$ for all $i$ as well. Finally, we may assume that $(1,0) \notin f(\Sigma(f))$.

We are going to define a smooth homotopy $f_t$ of stable mappings of
$S^1$ starting at $f$  and ending at $g$. This will furnish this step
because all of the $f_t$ will be smoothly equivalent. The standard
technique for producing homotopies between mappings in Euclidean space
by taking convex combinations of the mappings is not applicable here,
since $S^1$ is not a vector space. Nevertheless, by choosing appropriate
charts, the same strategy may be applied to coordinate neighbourhoods,
and our assumptions on $f$ and $g$ ensure that the resulting mapping
is in fact a smooth homotopy. The details are as follows. 

Let $n=n(f)=n(g)$, and let $N=N(f)=N(g)$. Let $\tau\in S_n$ be such
that $b_i=b_i(f)=P^{-1}(\sigma_{\tau(i)}(f))$. Notice that $b_1>0$
by the assumption $(1,0)\notin \Delta(f)$. Let 
\[0<v<\frac 12 \min_i (b_{i+1}-b_i,b_1,2\pi-b_n)\] and define
\[ \Theta_i:P(b_i-v,b_{i+1}+v)\to(-v,b_{i+1}-b_i+v), \quad P(x)\mapsto x-b_i\]
for $i=1,\ldots,n-1$. For $i=n$ we define
\[\Theta_n:S^1\setminus P([b_1+v,b_n-v]) \to (-v,b_1+2\pi -b_n+v)\]
by
\[ 
P(x)\mapsto 
\begin{cases}
 2\pi-b_n+x, & x\in[0,b_1+v)\\
 x-b_n,& x\in(b_n-v, 2\pi).
 \end{cases}
 \]
 The mappings $\Theta_i, i=1,\ldots, n$ are well defined by the
 choice of $v$.  Together with their domains of definition, they cover $S^1$ with local charts. 
 
 Let 
 \[0<u<\frac 12 \min_i(a_{i+1}-a_i, a_1, 2\pi-a_N)\]
and let $U_i=P(a_i-u,a_{i+1}+u), i=1,\ldots, N-1$ and
$U_N=S^1\setminus P([a_1+u, a_N-u])$. In the same way, let
$V_i=P(b_i-v,b_{i+1}+v), i=1,\ldots, n-1$ and
$V_n=S^1-P([b_1+v,b_n-v])$. We can now define our homotopy. 
By continuity of $f$ and $g$, if $v$ is small enough, then for
all $i$ there is a $j$ such that both $f(U_i)$ and $g(U_i)$ are
contained in $V_j$. More precisely; there exists $\rho:\{1,\ldots,N\}
\to \{1,\ldots,n\}$ such that for all $i$, $f(U_i)\cup g(U_i) \subset
V_{\rho(i)}$. For even smaller $u$, we can ensure that $\mbox{cl}(f(U_i)\cup
g(U_i))\subset V_{\rho(i)}$. Let $F:S^1\times(-\epsilon,1+\epsilon)\to S^1$ be defined by
\[
F(p,t)=f_t(p)=\Theta_{\rho(i)}^{-1}(t\Theta_{\rho(i)}(g(p))+(1-t)\Theta_{\rho(i)}(f(p))), \quad p\in U_i.
\]
We need to show that $f_t(p)$ is well defined on $S^1$ and that
$f_t(p)$ is smooth. The continuity of $\Theta_{\rho(i)}$ and the observation
that $t\Theta_{\rho(i)}(g(p))+(1-t)\Theta_{\rho(i)}(f(p))$ lies
between $\Theta_{\rho(i)}(g(p))$ and $\Theta_{\rho(i)}(f(p))$ for
$0\le t\le 1$, shows that $f_t$ is well defined on $U_i$ when
$\epsilon$ is chosen small enough.  

Next we show that the definitions of $f_t$ agree on $U_i\cap U_j$. It is enough to check the combinations $(i,j)=(N,1)$ and $(i,j)=(i,i+1)$ for $i<N$. The other combinations of $i$ and $j$ give $U_i\cap U_j=\emptyset$. We first assume that $1\le \rho(i)=\rho(j)-1<n$. Writing out the definitions,
\[
\begin{aligned} 
&\Theta_{\rho(i)}^{-1}(t\Theta_{\rho(i)}(g(p))+(1-t)\Theta_{\rho(i)}(f(p)))\\
=&\Theta_{\rho(i)}^{-1}(t[P^{-1}(g(p))-b_{\rho(i)}]+(1-t)[P^{-1}(f(p))-b_{\rho(i)}])\\
=&\Theta_{\rho(i)}^{-1}(t[P^{-1}(g(p))]+(1-t)[P^{-1}(f(p))]-b_{\rho(i)})\\
=&P(t[P^{-1}(g(p))]+(1-t)[P^{-1}(f(p))])\\
=&\Theta_{\rho(j)}^{-1}(t\Theta_{\rho(j)}(g(p))+(1-t)\Theta_{\rho(j)}(f(p))).
\end{aligned}
\]
For $\rho(i)=n$ and $\rho(j)=1$, we have
\[
\begin{aligned} 
&\Theta_1^{-1}(t\Theta_1(g(p))+(1-t)\Theta_1(f(p)))\\
=& P(t[P^{-1}(g(p))]+(1-t)[P^{-1}(f(p))])
\end{aligned}
\]
and
\[
\begin{aligned} 
&\Theta_n^{-1}(t\Theta_n(g(p))+(1-t)\Theta_n(f(p)))\\
=&\Theta_n^{-1}(t[P^{-1}(g(p))+2\pi-b_n]+(1-t)[P^{-1}(f(p))+2\pi-b_n])\\
=&\Theta_n^{-1}(t[P^{-1}(g(p))]+(1-t)[P^{-1}(f(p))]+2\pi-b_n)\\
=& P(t[P^{-1}(g(p))]+(1-t)[P^{-1}(f(p))]).
\end{aligned}
\]
This shows that $f_t$ is well defined on $S^1$ in this case. the case
$1<\rho(i)=\rho(j)+1\le n$ and the case $\rho(i)=1,\rho(j)=n$ may be
checked in a similar way. 

It remains to show that $f_t$ has finitely many singularities, all of Morse type, and no singular double points. In fact, $f_t$ has the same singular set and discriminant set as $f,g$. To actually show this, we need to work with charts in the source too. Let
\[ \theta_i:U_i\to(-u,a_{i+1}-a_i+u), \quad P(x)\mapsto x-u_i\]
for $i=1,\ldots,N-1$. For $i=N$ we define
\[\theta_N:U_N \to (-u,a_1+2\pi -a_N+u)\]
by
\[ 
P(x)\mapsto 
\begin{cases}
 2\pi-a_N+x, & x\in[0,a_1+u)\\
 x-a_N,& x\in(a_N-u, 2\pi).
 \end{cases}
 \]
The charts $(\theta_i,U_i)$ cover $S^1$.  Now, we may compute
\begin{equation}\label{E:s1}
\Theta_{\rho(i)}\circ f_t \circ \theta_i^{-1}(x)=t\Theta_{\rho(i)}(g(\theta_i^{-1}(x)))+(1-t)\Theta_{\rho(i)}(f(\theta_i^{-1}(x))).
\end{equation}
By our assumptions, $f$ and $g$ are equally oriented at every regular
point, and therefore the derivatives
with respect to $x$ of two terms on
the right side of \eqref{E:s1} have the same sign, and hence,
$\Sigma(f_t)=\Sigma(f)=\Sigma(g)$. Moreover, by definition of Morse
singularities, we must have
\[
\frac {d^2}{dx^2}\Theta_{\rho(i)}(f(\theta^{-1}(x)))\ne 0 \quad
  \mbox{ and }\quad \frac {d^2}{dx^2}\Theta_{\rho(i)}(g(\theta^{-1}(x)))\ne 0
\]
whenever $\theta^{-1}(x)\in \Sigma(f)$ and these second derivatives
must have the same sign at singular points. It follows that in these charts, the
second derivative of $f_t$ with respect to $x$ is different from 0 at
every singular point. Therefore, $f_t$ has only Morse
singularities. From the definiton of $f_t$, we see that $f(p)=g(p)$
implies that $f(p)=g(p)=f_t(p)$. It follows that $f_t$ has no singular
double points, and hence, $f_t$ is stable.

\emph{Step 2}. Assume that $\as(f)=\rho\cdot \as(g)$ for some $\rho =
\sigma\cdot \tau\in L_{N(g)}$, where $\sigma\in Sw_{N(g)}$, $\tau\in
R_{N(g)}$.  If $\tau = \mbox{id}$, then there is some
$\theta = \theta(\sigma)$ such that if \[\map {R_{\sigma}}{S^1}{S^1}\]
is given by \[e^{i\theta}\mapsto e^{i(\theta+\theta(\sigma))},\] then
\[\as(f)=\as(g\circ R_{\sigma})\] and by Step 1 there are
diffeomorphisms $h$ and $k$ such that the following diagram commutes.
\[
\xymatrix{
S^1\ar@/^10pt/[drr]^g& &\\
S^1\ar[u]^{R_{\sigma}}\ar[rr]^{g\circ R_{\sigma}}&&S^1\\
S^1\ar[u]^h\ar[rr]^f&&S^1\ar[u]_k}
\]
Similarly, if $\sigma =\mbox{id}$ and $\tau = r\in R_{N(g)}$, then, if \[\map
M{S^1}{S^1}\] is given by \[e^{i\theta}\mapsto e^{-i\theta},\] then
\[\as(f)=\as(g\circ M)\] and by Step 1 again, we have a commutative diagram:
\[
\xymatrix{
S^1\ar@/^10pt/[drr]^g& &\\
S^1\ar[u]^{R_{M}} \ar[rr]^{g\circ M}&&S^1\\
S^1\ar[u]^h\ar[rr]^f&&S^1\ar[u]_k}
\]
If $\rho = \sigma\circ r$ for some $\sigma \in Sw_{N(g)}$, then
\[
\as(f) = \sigma\cdot \as(g\circ M) = \as(g\circ M\circ R_{\sigma}),
\] 
which again, by the above arguments, implies that $f \sim_{\mathcal
  A_{\infty}} g$. Altogether we have shown that
$[\as(f)]_E=[\as(g)]_E\Rightarrow f\sim_{\mathcal A_{\infty}}g$.

\emph{Step 3}.
Suppose that $f$ and $g$ are $\mathcal A_{\infty}$-equivalent. Then there are
diffeomorphisms $h$ and $k$ of $S^1$ such that $k\circ f=g\circ
h$. Since a singularity of Morse type is topologically different from
a regular germ, it 
is clear that $h$ maps $\Sigma(f)$ to $\Sigma(g)$, and that
$k$ maps $\Delta(f)$ to $\Delta(g)$, and it follows
that $f^{-1}(\Delta(f))$ is mapped onto $g^{-1}(\Delta(g))$ by $h$,
and hence, $N(f)=N(g)$. If $h$ is orientation preserving and
$h(s_1(f))=s_i(g)$, then $\as(g)=\rho\cdot \as(f)$ where
$\rho([j])=[j+i-1]$. If $h$ is orientation reversing, then $\as(g)=\rho'\cdot \as(f)$ where
$\rho'([j])=[i-j+1]$. It follows that $[\as(f)]_E=[\as(g)]_E$.  
\end{proof}

\subsection{Feasible tuples}
By Theorem \ref{T:MT1}, the problem of listing all topological
equivalence classes of smooth stable maps $S^1\to S^1$ corresponds to
the problem of listing all $E_n$-equvalence classes of associated
tuples to such maps. Every non-singular map $f:S^1\to S^1$ is clearly equivalent to
the map $e^{i\theta}\mapsto e^{in\theta}$ where $n=\#f^{-1}(1)$. To
generate such a list for maps with singularities, we will make use of
another version of our tuples. If $f:S^1\to S^1$ is a smooth stable
map with $\Sigma(f)\ne \emptyset$, then $[\as (f)]_E$ may be represented by a tuple in
$\{s,p\}^{N(f)}$ 
having an $s$ as the last component. This may be done in several
different ways. Let $A'$ be such a representation. Let
$\rho:\{1,\ldots,n(f)\}\to \{1,\ldots,N(f)\}$ be such that $\rho$ is
increasing and $A'_{\rho(i)}=s$. Set $\rho(0)=0$ and let
$c_i=\rho(i)-\rho(i-1)-1$ for $i=1,\ldots, n(f)$. Define
\[
\as^{\#}(A')=(c_1,c_2,\ldots, c_{n(f)})\in \mathbb N_0^{n(f)}.
\]
Let 
\[
\as^{\#}(f)=[\as^{\#}(A')]_E
\]
 where $A'_{n(f)}=s$. It is not difficult to see that this definition
 of $\as^{\#}(f)$ is unambigious.

\begin{remark}\label{R:1}
Clearly, Theorem \ref{T:MT1} is still valid for maps with singularities if we replace $\as$ with
$\as^{\#}$. 
\end{remark}

Given an element $(x_1,x_2,\ldots, x_n)\in \mathbb N_0^n$, we want to
determine whether or not there is a smooth stable map $f:S^1\to S^1$
such that 
\[
\as^{\#}(f)=[(x_1,\ldots, x_n)]_E.
\]
Let $f$ be a stable map with $\as^{\#}(f)=[(x_1,\ldots, x_n)]_E$.
We say that $f$ is of type $(n,m)$ if $n(f)=n$ and $N(f)-n(f)=m$.
Thus, if $f$ is of type
$(n,m)$, then $n$ is an even number and
\[
x_1+x_2+\cdots +x_n=m.
\]
These two properties arise from observing that $f$ has an even number
of singular points, and that $N(f)-n(f)$ is the number of regular
preimage points of the discriminant set. Another property of $f$ is
that $f$
restricted to its singular set is injective, and this fact should be
reflected in $[(x_1,\ldots, x_n)]_E$. Indeed, the curve $P(x)$,
$x\in [0,2\pi)$, passes $x_i$ points in $f^{-1}(\Delta(f))\setminus
\Sigma(f)$ when $x$ runs through $I=[P^{-1}(s_{i-1}),P^{-1}(s_i))$. Therefore, the curve
$f(P(x))$ passes $x_i$ singular values in the same interval of
parameters. Thus, if $\sigma_{i-1}=P(b_j)$ and $f$ is orientation
preserving on $P(\mbox{int } I)$ and 
\[
k=(\mbox {remainder of the division $x_i$ by $n$})+1,
\]
then $\sigma_i=P(b_{j+k})$. 

In general, let $R:\mathbb Z\to
\{1,2,\ldots n\}$ be given by $R(x) = 
(\mbox{remainder of the division $x$ by $n$})+1$. Let $\tau\in S_n$ be as in
the proof of Theorem \ref{T:MT1}, i.e. such that $P(b_i)=\sigma_{\tau(i)}$. Assuming that
$(x_1,x_2,\ldots,x_n)=\as^{\#}(\as(f))$, $\sigma_1=P(b_j)$ and that
$f$ is orientation reversing on $P(P^{-1}(s_1),P^{-1}(s_2))$, then we see that
\[
\sigma_k=\sigma_{\tau(R(j-x_1-1+\sum_{i=1}^k(-1)^{i+1}[x_i+1]))}
\]
for $k=1,\ldots, n$. Moreover, since we chose representatives with $s$
in the last component in the definiton of $\as^{\#}$, we have
\[
\sigma_n=P(b_{R(j-x_1-1)}).
\]
In order for all these equations to be satisfied, the set 
\[
R'=\left\{ \sum_{i=1}^k(-1)^{i+1}[x_i+1]\,;\, k=1,\ldots,n\right\}
\]
has to be a complete remainder system modulo $n$, i.e., the canonical
map $R'\to\mathbb Z/n\mathbb Z$ is surjective. Furthermore, 
\[
 \sum_{i=1}^n(-1)^{i+1}[x_i+1]\equiv 0\mod n.
\]

\begin{definition}
An element $A=(x_n,\ldots,x_n)\in \mathbb N_0^n$ is feasible of type
$(n,m)$ if $n$ is an even number and the following condtions are satisfied:
\begin{enumerate}
\item $\sum_{i=1}^n x_i =m$.
\item $\sum_{i=1}^n (-1)^{i+1}x_i\equiv 0\mod n$.
\item $\left\{ \sum_{i=1}^k(-1)^{i+1}[x_i+1]\,;\, k=1,\ldots,n\right\}$
is a complete remainder system modulo $n$.
\end{enumerate}
\end{definition}

\begin{remark}
There are no feasible tuples of type $(n,m)$ if $m$ is odd, because
the numbers $m=\sum_{i=1}^n x_i$ and $\sum_{i=1}^n
(-1)^{i+1}x_i$ have the same parity, and by 2 in the definition, the
latter number is even, since $n$ is even.
\end{remark} 

\begin{proposition}\label{P:mn4}
There are no feasible tuples of type $(n,m)$ if $n\equiv 0 \mod 4$ and
$m\equiv 2\mod 4$.
\end{proposition} 
\begin{proof}
Assume that $(x_1,\ldots, x_n)$ is feasible of type $(n,m)$. 
Let $L_k=\sum_{i=1}^k (-1)^{i+1}(x_i+1)$. Notice that 
\[
2\left(\sum_{i=1}^{n-1}(-1)^{i+1}L_i\right) -L_n = x_1+\cdots+x_n+n=m+n.
\]
Since $L_n\equiv 0 \mod n$,
\begin{equation}\label{E:mod1}
2\sum_{i=1}^{n-1}(-1)^{i+1}L_i\equiv m \mod n.
\end{equation}
Since $\{L_k;k=1,\ldots,n\}$ is a complete remainder system modulo
$n$, we have
\begin{equation}\label{E:mod2}
2\sum_{i=1}^nL_i \equiv 2\sum_{i=0}^{n-1} i \equiv n(n-1) \equiv 0
\mod n.
\end{equation}
Addition of \eqref{E:mod1} and \eqref{E:mod2} yields
\begin{equation}\label{E:mod3}
4(L_1+L_3+L_5+\cdots + L_{n-1}) \equiv m \mod n.
\end{equation}
Hence, there is an integer $K$ such that 
\[
4(L_1+L_3+L_5+\cdots + L_{n-1})-m=Kn.
\]
It follows that $4|n\Rightarrow 4|m$.
\end{proof}

The next theorem justifies the  term 'feasible tuple'.
\begin{theorem}\label{T:tup1}
Let $A\in \mathbb N_0^n$. There exists a smooth stable map $f:S^1\to
S^1$ with $\as^{\#}(f)=A$ if and only if $A$ is feasible of type
$(n,m)$ for some number $m$.
\end{theorem}
\begin{proof}
The forward implication follows from the above discussion. For the
other implication, let $A=(x_1,\ldots,x_n)\in \mathbb N_0^n$ be a feasible tuple of
type $(n,m)$. We need to construct a smooth stable map $f:S^1\to
S^1$ with $\as^{\#}(f)=A$. We construct a smooth map $f_A:[0,2\pi)\to \rr$ such
that $f=P\circ f_A\circ P^{-1}$ is smooth and stable and
satisfies $\as^{\#}(f)=A$. It is natural to define $f_A$ to
consist of linesegments ouside some small open intervals about the
singular points and consist of a modified parabel around the singular
points. This strategy calls for some kind
of gluing process, but we can not use a standard partition of unity,
because we must have full control over the singularities of $\tilde
f$, and a partition of unity might introduce unwanted
singularites. Instead, we will construct $f_A$ explicitly, using
smooth ``bump functions'' to glue the different parts of the function together.

Let 
\[
j(x)=\begin{cases}e^{-(x-1)^{-2}}\cdot e^{-(x+1)^{-2}},&x\in
(-1,1)\\\phantom{e^{-(x-1)^{-2}}} 0,&\mbox{otherwise}\end{cases}
\]
and let
\[
k(x)=\frac{\int_{-1}^xj(t)dt}{\int_{-1}^1j(t)dt}.
\]
Define
\[
l(x)=\begin{cases}x,&x\leq -1\\x-2xk(x),&x\in (-1,1)\\-x,&x\geq 1.\end{cases}
\]
Then
\[
l'(x)=\begin{cases}1,&x\leq -1\\1-2k(x)-2xk'(x),&x\in (-1,1)\\-1,&x\geq 1,\end{cases}
\]
and
\[
l''(x)=\begin{cases}0,&x\leq -1\\-4k'(x)-2xk''(x),&x\in (-1,1)\\0,&x\geq 1,\end{cases}
\]
Since $k$ is flat at $-1$ and $1$, $l$ is a $C^{\infty}$ function on
$\rr$. Also, $l$ is increasing for $x\le 0$ and decreasing for $x\geq
0$. Since $l'(0)=0$ and $l''(0)=-4k'(0)<0$, this means that $l$ has
its only extreme point at $x=0$ and this is a global maximum and a
Morse singularity. The definition of $f_A$ is the following. For
$k=1,\ldots, n$, let
\[
\begin{aligned}
X_k&=\sum_{i=1}^k(x_i+1)\\
Y_k&=\sum_{i=1}^k(-1)^{i+1}(x_i+1)\\
J_k&=[X_k-\frac 12,X_k+\frac 12)
\end{aligned}
\]
Let 
\[
I_0=[0,X_1-\frac 12)
\]
and for $k=1,\ldots,n-1$ let
\[
I_k=[X_k+\frac 12,X_{k+1}-\frac 12).
\]
For a set $B\in \rr$, let $\chi_B$ be the corresponding characteristic
function which is 1 on $B$ and 0 elsewhere. Put $X_0=Y_0=0$. For $k=1,\ldots n$,
let
\[
\begin{aligned}
F_k(x)&=\left[Y_{k-1}+(-1)^{k-1}(x-X_{k-1})\right]\chi_{I_{k-1}}\\
G_k(x)&=\left[Y_k+\frac{(-1)^{k+1}}2l(2(x-X_k))\right]\chi_{J_k}.
\end{aligned}
\]
Let 
\[
H(x)=\sum_{i=1}^n(F_k(x)+G_k(x)).
\]
Finally, let 
\[
f_A(x)=\frac{2\pi}nH\left(\frac{X_n}{2\pi}x+\frac 12\right).
\]
With this definition of $f_A$, let $f=P\circ f_A\circ P^{-1}$.
It is messy, but straight forward to see that $f$ is smooth and
that $\as^{\#}(f)=A$.
\end{proof}

Let $f$ be a smooth stable map of the circle. All the topological
properties of $f$ is coded in $\as^{\#}$. We show how $|\deg f|$ can be retrieved from $\as^{\#}(f)$.

\begin{proposition}\label{P:deg}
Let $\as^{\#}(f)=(x_1,x_2,\ldots,x_n)$. Then
\[
|\deg f| =\left|\frac 1n \sum_{i=1}^n (-1)^{i+1}x_i\right|.
\]
\end{proposition}

\begin{proof}
Let $A=\as^{\#}(f)$, and let $f_A$ be as in the proof of Theorem
\ref{T:tup1}. Then $\deg(f)=\deg(p\circ f_A\circ P^{-1})$. Certainly,
$f_A$ is homotopic to $\tilde f_A$ given by 
\[
\tilde f_A(x)=\frac{1}n\left(\sum_{i=1}^n (-1)^{i+1}x_i\right)x.
\]
by the homotopy $F(x,t)=tf_A(x)+(1-t)\tilde f_A(x)$. Clearly, $p\circ
\tilde f_A\circ P^{-1}$ has degree $\frac 1n\sum_{i=1}^n
(-1)^{i+1}x_i$, and this finishes the proof.
\end{proof}


\subsection{Tables of feasible tuples}
A complete classification of smooth stable maps $S^1\to S^1$ can be
given by listing all the feasible tuples up to legal
permutations. This task is well suited for recursive computer
programming. Table \ref{Ta:Types1} and Table \ref{Ta:Types2} give
MATLAB generated 
lists of feasible tuples and numbers of topological types for
different $(n,m)$. 

\begin{table}[!]
\begin{center}
\begin{tabular}{|l|c|l|}\hline
$(n,m)$&\textbf{Number of topological types} &\textbf{Feasible tuples} \\ \hline\hline
$(4,4)$ & 2& $\begin{aligned}(1,2,1,0),\,(2,0,2,0)\end{aligned}$\\
  \hline
$(4,8)$ & 5& $\begin{aligned}&(5,2,1,0),\,(1,6,1,0),\,(2,4,2,0),\\&(6,0,2,0),\,(4,1,2,1)\end{aligned}$\\
  \hline
$(4,12)$ &12 & $\begin{aligned}&(9,2,1,0),\,(5,6,1,0),\,(1,10,1,0),\\&(6,4,2,0),\,(2,8,2,0),\,(5,2,5,0),\\&(8,1,2,1),\,(4,5,2,1),\,(6,1,4,1),\\&(10,0,2,0),\,(6,0,6,0),\,(4,2,4,2)\end{aligned}$\\
  \hline
$(6,6)$&1&$\begin{aligned}(2,0,2,0,2,0)\end{aligned}$\\\hline
$(6,8)$&2&$\begin{aligned}(3,1,0,3,1,0),\,(2,0,1,4,1,0)\end{aligned}$\\\hline
$(6,10)$&3&$\begin{aligned}&(3,0,4,2,1,0),\,(1,4,0,4,1,0),\\
&(3,1,2,1,3,0)\end{aligned}$\\\hline

\end{tabular}
\end{center} 
\caption{Table of topological types}\label{Ta:Types1}
\end{table} 

\begin{table}[!]
\begin{center}
\begin{tabular}{|c|c||c|c|}\hline
$(n,m)$&\textbf{Number of topological types} &$(n,m)$&\textbf{Number
    of topological types} \\ \hline\hline
$(4,16)$&21&$(8,8)$ & 1\\\hline
$(4,20)$&36&$(8,12)$ &12\\\hline
$(4,24)$&54&$(8,16)$ &34\\\hline
$(4,28)$&80&$(10,10)$ &1\\\hline
$(6,12)$ & 9& $(10,12)$&0\\\hline
$(6,14)$ & 10& $(10,14)$&3\\  \hline
$(6,16)$ &16 & $(10,16)$&6\\\hline
\end{tabular}
\end{center} 
\caption{Number of topological types}\label{Ta:Types2}
\end{table} 

Our tables lack the number of feasible tuples of type $(2,m)$ because
of the next proposition.

\begin{proposition}\label{P:top}
The number of $E_2$-equivalence classes of feasible tuples of type
$(2,m)$ is $\lfloor \frac m4\rfloor + 1$.
\end{proposition}
\begin{proof}
Assume $(x_1,x_2)$ is feasible of type $(2,m)$. Then
\[
\begin{aligned}
x_1+1&\equiv 1 \mod 2\\
x_1-x_2&\equiv 0 \mod 2.
\end{aligned}
\]
These equations are satisfied if and only if $x_1$ is even and
$x_1$ and $x_2$ have the same parity. The feasible tuples of type
$(2,m)$ are therefore $\{(2i,m-2i); i=0,1,\ldots, \frac m2\}$. There
are $\frac m2+1$ elements in this set, and $(2i, m-2i)\sim_{E_2}
(m-2i,2i)$ for all $i$. If $m=4k$ for some $k\in \mathbb N$, then
$\frac m2+1=2k+1$ is odd, and the number of $E_2$-equivalence classes
is $k+1=\lfloor \frac m4 \rfloor+1$. If $m=4k+2$, then $\frac m2+1=
2k+2$ is even, and the number of equivalence classes is still
$k+1=\lfloor \frac m4\rfloor +1$.
\end{proof}


\section{Classification of finitely determined real analytic map germs $(\re
  2,0)\to (\re 2,0)$}\label{S:class2}

Let $\mathcal O=\mathcal O(2,2)$ be the set of real analytic map germs $(\re
  2,0)\to (\re 2,0)$. Let $\mathcal O_{\mbox g}=\mathcal O_{\mbox g}(2,2)\subset
  \mathcal O(2,2)$   be the set of finitely determined map germs. By
  Theorem 0.5 of \cite{duPlessis}, finite determinacy holds in general
  in $\mathcal O(2,2)$. 

\subsection{Geometric properties}\label{S:geo}
Finitely determined real analytic plane-to-plane germs have the following 
well known geometric properties.

\begin{proposition}\label{P:1}
For every $f\in \mathcal O_{\mbox g}$ there is an open neighbourhood $U$ of
0 in $\re 2$ and a real analytic representative of $f$, $\hat f:U\to
\re 2$ such that 
\begin{enumerate}
\item $\hat f^{-1}(0) = \{0\}$,
\item $\hat f|(\Sigma(\hat f)\setminus \{0\})$ is injective,
\item every $p\in\Sigma(\hat f)\setminus \{0\}$ is a fold point.
\end{enumerate}
\end{proposition}
\begin{proof}
The proof of (2) and (3) goes as the proof of Lemma 6.2 in
\cite{BroSku} with semianalytic substituted for semialgebraic. To
prove (1), note that $\hat f^{-1}(0)\setminus \{0\}$ is a semianalytic
set. If 0 is in its closure, then by the Curve Selection Lemma, there
is a real analytic curve $\gamma:[0,\epsilon)\to \re 2$ with
$\gamma(0)=0$, $\gamma(0,\epsilon)\in \hat f^{-1}(0)\setminus
\{0\}$. Hence, $\hat f$ is identically 0 along $\gamma$, but this
contradicts both (2) and (3).
\end{proof}

For the rest of this section, let $f\in \mathcal O_{\mbox g}$, let
$U$ be a small ball around 0 and let
$\hat f:U\to \re 2$ be a real analytic representative of $f$ such that
(1)-(3) of Proposition \ref{P:1} hold.

\begin{lemma}\label{L:fini}
If $U$ is small enough, then $\Sigma(\hat f)\setminus \{0\}$ is empty
or a 1-dimensional
manifold which has only finitely many
topological components. 
\end{lemma}
\begin{proof}
By (3), if $p\in\Sigma(\hat f)\setminus \{0\}$, then $p$ is a fold
 point, and the singular set is diffeomorphic to the real line in a
 neighbourhood of a fold point. Also, $\Sigma(\hat f)\setminus \{0\}$ is a semianalytic set, and
 hence, its intersection with a small neighbourhood of $0$ has only
 finitely many topological components.
\end{proof}
 
Let $D_{\epsilon}= \{p\in \re 2 \mid
\en p\le\epsilon\}$ and let $S_{\epsilon}=\{p\in \re 2 \mid
\en p=\epsilon\}=\partial D_{\epsilon}$. 
Define $\tilde S_{\epsilon}(\hat f)= \hat f^{-1}(S_{\epsilon})$ and
$\tilde D_{\epsilon}(\hat f)=\hat f^{-1}(D_{\epsilon})$. 

\begin{lemma}\label{L:1}
If $U$ is small enough, then $\hat f\pitchfork S_{\delta}$ for small enough $\delta>0$.
\end{lemma}
\begin{proof}
By Lemma \ref{L:fini} there are only finitely many
branches of $\Sigma(\hat f)\setminus\{0\}$. By the Curve Selection Lemma,
for each component $B_i$ of $\sigfn$ 
we may choose an analytic curve $\gamma_i:[0,\epsilon)\to \re 2$ such
that $\gamma(0)=0$ and $\gamma_i(0,\epsilon)\subset B_i$.
The curves $\hat f\circ \gamma_i$ are
analytic and by (1) of Proposition \ref{P:1}, $(\hat f\circ
\gamma_i)(t)\ne 0$ when $t>0$ and therefore $(\hat f\circ \gamma_i)\pitchfork S_{\delta_i}$ for
small $\delta_i>0$. If $\delta < \underset i{\mbox{min }} \delta_i$,
then $\hat f|\Sigma(\hat f)\pitchfork S_{\delta}$. This proves the lemma, since
$\hat f\pitchfork S_{\delta}$ at any regular point of $\hat f$ because
the dimensions of source and target are equal.
\end{proof}
The proof of Lemma \ref{L:1} actually gives us more information. Let
$\Delta (\hat f)=\hat f(\Sigma(\hat f))$.

\begin{corollary}\label{C:tran}
If $U$ is small enough, then $\Delta(\hat f)\setminus \{0\}$ is empty
or a real analytic and such that $\Delta(\hat f)\pitchfork S_{\delta}$ for small $\delta$.
\end{corollary}
Let $\theta:\re 2\to \rr$ be given by $\theta(p)=\en p^2$.
\begin{lemma}\label{L:grad}
If $\delta$ is small enough, then $\nabla(\theta\circ \hat f)(p)\ne 0$
for all $p\in D_{\delta}\setminus \{0\}$.
\end{lemma}
\begin{proof}
If $\hat f=\left(\begin{smallmatrix}f_1\\f_2\end{smallmatrix}\right)$, then
$\theta \circ \hat f$ = $f_1^2+f_2^2$. We compute 
\[
\begin{aligned}
\nabla(\theta \circ \hat f)(p)&=2\big(f_1\der {f_1}x+f_2\der
{f_2}x,f_1\der {f_1}y+f_2\der {f_2}y\big)(p)\\
&=2\left(\begin{matrix}f_1(p)&f_2(p)\end{matrix}\right)\cdot D\hat f(p)
\end{aligned}
\]
If $p\notin \Sigma(\hat f)$, then $\hat f(p)\ne 0$ and $D\hat f(p)$ is
invertible, and hence, $\nabla (\theta\circ \hat f)(p)\ne 0$. Assume
that $p\in \Sigma(\hat f)$ and $\en p\ne 0$. By (1), $\hat f(p)\ne 0$
and by the above,
\[
\begin{aligned}
\nabla(\theta \circ \hat f)(p)=0&\Leftrightarrow \hat f^T(p)D\hat f(p)=0\\
&\Leftrightarrow \hat f(p)\perp \mbox{Im\,}D\hat f(p)\\
&\Leftrightarrow D\hat f(p)(\re 2)+\rr \{\left(\begin{matrix}-f_2(p)\\
    \phantom - f_1(p)\end{matrix}\right)\} \ne \re 2.
\end{aligned}
\]
Note that $\left(\begin{matrix}-f_2(p)\\
    \phantom - f_1(p)\end{matrix}\right)$ is a tangent vector at $\hat
f(p)$ to the circle $S_{\en{\hat f(p)}}$. It therefore
follows from Lemma \ref{L:1} that $D\hat f(p)(\re 2)+\rr \{\left(\begin{matrix}-f_2(p)\\
    \phantom - f_1(p)\end{matrix}\right)\} =\re 2$. This proves the lemma.
\end{proof}

\begin{lemma}[Lojasiewicz]\label{L:Loj}
There is a $\rho>0$  and  
constants $C,r>0$ such that for $p\in D_{\rho}$, $\enorm{\hat f(p)}\ge C\enorm
p^r$.
\end{lemma}
\begin{proof}
Remember that 0 is an isolated zero of $\hat f$ and apply IV 4.1 of
\cite{Malgrange}. 
\end{proof}
\begin{lemma}\label{L:2}
For small $\epsilon>0$, $\tse {\hat f}$ is a compact
1-manifold diffeomorphic to $S^1$ and $0$ is in the bounded component
of $\re2\setminus\tse {\hat f}$.
\end{lemma}
\begin{proof}
Let $\rho>0$ be such that $\enorm{\hat f(p)} \ge C\enorm p^r$
for all $p\in D_{\rho}$. Such a $\rho$ exists by Lemma \ref{L:Loj}. If $\epsilon\le C\rho^r$,
then $\tilde S_{\epsilon}\subset D_{\rho}$ is closed and
bounded, i.e. compact. By Lemma \ref{L:1}, if $\rho$ is small enough,
then $f\pitchfork S_{C\rho^r}$ in which case $\tilde S_{\epsilon}$
is a 1-dimensional smooth manifold. 

Every component of $\tilde S_{\epsilon}$ is diffeomorphic to $S^1$ by
the classification of smooth compact 1-manifolds. Let $C$ be one such
component. Then $C$ is an equipotensial curve of $\theta\circ \hat
f$. If $0$ is not in the bounded component of $\re 2\setminus
C$, then $\theta\circ \hat f$ has an extremal point $p$ in the bounded
component of $\re 2\setminus C$, and hence, $\nabla (\theta\circ \hat
f)(p)=0$. According to Lemma \ref{L:grad}, this is not possible for
small $\rho$. It follows that 0 is in the bounded component of $\re
2\setminus C$. 

Assume that $C$ and $D$ are different components of $\tse{\hat
  f}$. Then there are two bounded components of $\re 2\setminus (C\cup
D)$, one of the containing 0. The other component must contain an
extremal point of $\theta\circ \hat f$ which is impossible for small $\rho$.
\end{proof}

Figure \ref{F:2} below illustrates some of the properties we have proven so far.
\begin{figure}
\begin{center}
\resizebox{13cm}{!}{\input{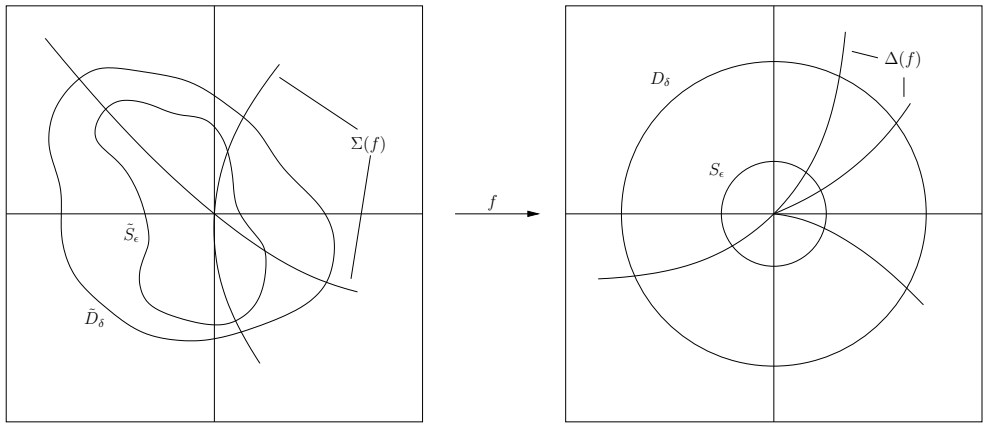}}
\caption{Illustration of Lemma \ref{L:fini}, Lemma \ref{L:1},
  Corollary \ref{C:tran} and Lemma \ref{L:2}.}
\label{F:2}
\end{center}
\end{figure}

Let $E_{\delta}=\{p\in \re 2\mid \en p<\delta\}=\mbox{int }
D_{\delta}$ and let $\tilde E_{\delta}(\hat f)=\hat f^{-1}(E_{\delta})$.

\begin{lemma}\label{L:3} 
For small $\delta>0$ the map $\hat f|\tilde E_{\delta}\setminus
\{0\}:E_{\delta}\setminus \{0\}\to E_{\delta}\setminus \{0\}$ is proper. 
\end{lemma} 
\begin{proof}
By Lemma \ref{L:Loj} there are $C,r,\rho>0$ such that $\en{\hat
  f(p)}\ge C\en p^r$ for all $p\in D_{\rho}$. Assume that $\delta$ is
so small that $\max\{\delta,\left(\frac{\delta}C\right)^{\frac 1r}\}<\rho$. Redefine
$\hat f$ putting $\hat f:=\hat f|E_{\rho}$. Then
$\tilde D_{\delta}(\hat f) \subset
D_{\left(\frac{\delta}C\right)^{\frac 1r}}\subset E_{\rho}$, and
hence, $\tilde D_{\delta}(\hat f)$ is compact.

Let $K\subset E_{\delta}\setminus \{0\}$ be a compact set. Let $\tilde
K=(\hat f|\tilde E_{\delta}\setminus \{0\})^{-1}(K)$ and let $(p_n)$
be a sequence in $\tilde K$. Then $(p_n)$ is a sequence in $\tilde
D_{\delta}(\hat f)$, and hence, there is a subsequence
$p_{n(k)}$ of $p_n$ and a point $p\in \tilde D_{\delta}(\hat
f)$ such that $p_{n(k)}\to p$ as $k\to \infty$. Then $\hat
f(p_{n(k)})\to \hat f(p)\in K$, and hence, $p\in
\tilde K$. It follows that $\tilde K$ is compact and that $\hat
f|\tilde E_{\delta}\setminus\{0\}$ is proper.
\end{proof}

\begin{proposition}\label{P:stab}
For small $\epsilon>0$, the
restriction $\hat f|\tse{\hat f}:\tse{\hat f} \to \sep$ is stable.
\end{proposition}
\begin{proof}
It is enough to show that $\hat f|\tse {\hat f}$ has only Morse singularities
and no singular double points. Corollary \ref{C:tran} implies that $\tse
{\hat f}\pitchfork \Sigma(\hat f)$ close to the origin. We also
observe that $\Sigma(\hat f|\tse
  {\hat f})\subset \Sigma(\hat f)$.  Let $p\in \Sigma(\hat f|\tse{\hat
    f})$, and let $\beta$ be a centered chart about $p$ in $\tse {\hat
    f}$, and
let $\pi$ be the projection of onto the line $L$
perpendicular to $\Delta(\hat f)$ at $\hat f(p)$. The restriction of $\pi$ to a
neighbourhood of $\hat f(p)$ in $S_{\epsilon}$ is a chart about $\hat f(p)$ in
$S_{\epsilon}$. Let $\Psi$ and $\Phi$ be diffeomorphisms of
neighbourhoods of $p$, $\hat f(p)$ in $U$, $\re 2$ respectively such that
$\hat f=\Phi\circ F\circ \Psi$ where $F(x,y)=(x,y^2)$. Such diffeomorphisms
exist because $p$ is a fold point of $\hat f$. Now, choose a linear
isomorphism $T:L\to \rr$ which identifies $L$ with $\rr$ such that
$T(\pi(\hat f(p)))=0$. 

Let $\alpha =(\alpha_1,\alpha_2)= \Psi\circ \beta^{-1}$ and let $A=T\circ \pi\circ
\Phi$. Then $\hat f|\tse{\hat f}\sim_{\mathcal
  A}A\circ F\circ \alpha$. Now we compute
\[
(A\circ F\circ \alpha)'(t)= A_x\alpha_1'(t)+2A_y\alpha_2(t)\alpha_2'(t)
\]
and
\[
\begin{aligned}
(A\circ F\circ
\alpha)''(t)=&[A_{xx}\alpha_1'(t)+2A_{xy}\alpha_2(t)\alpha_2'(t)]\alpha_1'(t)+A_x\alpha_1''(t)\\
&+[A_{yx}\alpha_1'(t)+2A_{yy}\alpha_2(t)\alpha_2'(t)]\cdot
2\alpha_2(t)\alpha_2'(t)\\
&+A_y[2(\alpha_2'(t))^2+2\alpha_2(t)\alpha_2''(t)].
\end{aligned}
\]
Here all the partial derivatives of $A$ are to be taken at $F\circ \alpha(t)$.
Since there is no neighbourhood of $p$ in $\tse{\hat f}$ restricted to
which $\hat f |\tse{\hat  f} $ is injective and since $\tse{\hat
  f}\pitchfork \Sigma(\hat f)$,
we see from the normal form $F$ of the folds that $\alpha_1'(0)=0$ and
$\alpha_2'(0)\ne 0$. We have also chosen $\alpha_2(0)=0$. The choice
of $L$ gives $A_x(F(\alpha(0))=0$. Therefore we must have
$A_y(F(\alpha(0)))\ne 0$. We get
\[
(A\circ F\circ \alpha)''(0) =
2A_y((F(\alpha(0)))\cdot(\alpha_2'(0))^2\ne 0.
\]
This shows that $A\circ F\circ \alpha$ has a Morse singularity at 0,
and hence, $\hat f|\tse  {\hat f}$ has a Morse singularity at $p$.
\end{proof}


\subsection{Generic mappings as cones of smooth stable mappings
  between spheres }\label{S:cone} 
In this section we follow the steps in \cite{Fuk1}
pp. 246-247. 

Let $f\in \mathcal O_{\mbox{g}}$ and let $\hat f:U\to \re 2$ be a
fixed representative of $f$ with $U$ so small that the lemmas of the
previous section hold. We simplify notation putting $\tilde
S_{\epsilon}:=\tilde S_{\epsilon}(\hat f)$ and similar simplifications
for $\tilde D_{\epsilon}(\hat f)$ and $ \tilde E_{\epsilon}(\hat f)$. Let $\delta$ be so small that  
 $\nabla (\theta\circ \hat f)\ne 0$ on
$\tilde D_{\delta}\setminus \{0\}$ and let $\epsilon,\alpha>0$ be such
that $\epsilon+\alpha <\delta$. 
Let $\varphi_p(t)$ be the flowline of $\nabla(\theta\circ \hat f)$ passing
through $p$, and let $t_p$ be such that $\varphi_p(t_p)\in \tilde
S_{\epsilon}$. Define maps 
\[
\begin{aligned}
&\map {\phi}{\tilde
  E_{\epsilon+\alpha}-\{0\}}{\tilde S_{\epsilon}},\\ 
&\map {\Phi}{\tilde E_{\epsilon+\alpha}-\{0\}}{\tilde
    S_{\epsilon}\times   (0,\epsilon+\alpha)},\\
&\map {\Psi}{E_{\epsilon+\alpha}-\{0\}}{S_{\epsilon}\times
  (0,\epsilon+\alpha)}
\end{aligned}
\]
by
\[
\begin{aligned}
\phi(p)&=\varphi_p(t_p)\\
\Phi(p)&=(\phi(p),\enorm{\hat f(p)})\\
\Psi(q)&=(\epsilon \frac q{\enorm q}, \enorm q)
\end{aligned}
\]
Both $\Phi$ and $\Psi$ are certainly diffeomorphisms, and we can
define \[\map F{\tilde S_{\epsilon}\times
  (0,\epsilon+\alpha)}{S_{\epsilon}\times (0,\epsilon+\alpha)}\] by
$F=\Psi\circ \hat f\circ \Phi^{-1}$. Then $F(\tilde S_{\epsilon}\times
\{t\})\subset S_{\epsilon}\times \{t\}$ and the following diagram
commutes.
\[
\begin{CD}
\tilde E_{\epsilon+\alpha}-\{0\}@>\hat f>>E_{\epsilon+\alpha}-\{0\}\\
@V\Phi VV@VV\Psi V\\
\tilde S_{\epsilon}\times (0,\epsilon+\alpha)@>F>>S_{\epsilon}\times
(0,\epsilon+\alpha)
\end{CD}
\]
Let $\map {f_t}{\tilde S_{\epsilon}}{S_{\epsilon}}$ be defined by
$F(p,t)=(f_t(p),t)$. Then $f_t$ is a smooth homotopy and
$f_{\epsilon}=\hat f|{\tilde S_{\epsilon}}$. If we let $\avb {\pi}2{}$ be
the projection onto the first factor, we get
\[
\begin{aligned}
f_t&=\pi\circ F|{\tilde S_{\epsilon}\times \{t\}}\\
&= \pi\circ \Psi\circ \hat f\circ \Phi^{-1}|{\tilde S_{\epsilon}\times
  \{t\}}\\
&= \pi\circ\Psi\circ \hat f|{\tilde S_t}\circ \Phi^{-1}|{\tilde S_{\epsilon}\times
  \{t\}}.
\end{aligned}
\]
Thus, $f_t$ is $C^{\infty}$ equivalent to $\hat f|{\tilde S_t}$. It
follows from Proposition  \ref{P:stab} that all $\hat f|{\tilde S_t}$ and
hence, every $f_t$ is smoothly stable. Hence, there are $C^{\infty}$ diffeomorphisms 
\[
\map {h_t'}{\tilde S_{\epsilon}}{\tilde S_{\epsilon}}
\]
and
\[
\map {h_t''}{S_{\epsilon}}{S_{\epsilon}}
\]
such that $\hat f|\tilde S_{\epsilon}\circ h_t' = h_t''\circ f_t$ and we can choose
$h_t'$ and $h_t''$ such that
$h_{\epsilon}'=\mbox{id}$ and $h_{\epsilon}''=\mbox{id}$ and the
mappings
\[
\map {H'}{\tilde S_{\epsilon}\times (0,\epsilon+\alpha)}{\tilde
  S_{\epsilon}\times(0,\epsilon+\alpha)}
\]
and
\[
\map {H''}{S_{\epsilon}\times (0,\epsilon+\alpha)}{
  S_{\epsilon}\times(0,\epsilon+\alpha)}
\]
defined by $H'(x,t)=(h_t'(x),t)$ and $H''(y,t)=(h_t''(y),t)$ are
diffeomorphisms. It follows that $\hat f|\tilde E_{\epsilon+\alpha}\setminus \{0\}\sim_{\mathcal A_{\infty}}
F=(f_t,\mbox{id})\sim_{\mathcal A_{\infty}} (\hat f|\tse {\hat f},\mbox{id})$.

 
\subsection{The main theorem}
According to Proposition \ref{P:stab}, if $f\in \mathcal O_{\mbox g}$,
then $\hat f|{\tilde
  S_{\epsilon}}:\tse {\hat f} \to \sep$ is stable for small $\epsilon$. Also,
  the homotopy   $f_t$ of Section \ref{S:cone} is a smooth homotopy of
  $C^{\infty}$ 
  stable mappings, and hence, they are all $C^{\infty}$
  equivalent. Therefore, regarding $\hat f|{\tilde S_{\epsilon}}$ as a map
  between 1-spheres, we can associate a tuple $\as(f)$ unambigously
  to $f$ by the rule $\as(f)=[\as(\hat f|{\tilde
  S_{\epsilon}})]_E$, the equivalence class of $\as(\hat f|{\tilde
  S_{\epsilon}})$ under the equivalence relation introduced in Section
  \ref{SS:gr}. In the same way, we define $\as^{\#}(f)=[\as(\hat f|{\tilde
  S_{\epsilon}})]_E$ when $\Sigma(f)\ne \{0\}$. It is clear that 
\[
\as^{\#}(f)=\as^{\#}(g)\Leftrightarrow \as(f)=\as(g).
\]
\begin{theorem}\label{T:MT}
If $f,g \in \mathcal O_{\mbox g}$ and $\Sigma(f)\setminus
\{0\},\Sigma(g)\setminus \{0\}\ne \emptyset$, then 
\[
f\sim_{\mathcal A_0}g
\Leftrightarrow \as(f)=\as(g).
\]
\end{theorem}
\begin{proof}
The latter equivalence is immediate from the definitions.
Assume $\as(f)= \as(g)$. Choose representatives $\hat f$ and $\hat g$
for $f$ and $g$ and construct the homotopies $f_t$ and
$g_t$ as in section 
\ref{S:cone}. Clearly, for small $\epsilon$ and $\alpha$, 
$\hat f|\tilde E_{\epsilon+\alpha}(\hat f)\setminus \{0\}\sim_{\mathcal A_{\infty}}
F=(f_{\epsilon},\mbox{id})$ and $\hat g|\tilde E_{\epsilon+\alpha}(\hat
g)\setminus \{0\}\sim_{\mathcal A_{\infty}}
G=(g_\epsilon,\mbox{id})$. Now, by hypothesis and Theorem \ref{T:MT1}, there
are suitable homeomorphisms
$k_{\epsilon}$ and $h_{\epsilon}$ (which can be chosen to be smooth) such 
that
\[
f_{\epsilon}=k_{\epsilon}\circ g_{\epsilon}\circ h_{\epsilon}^{-1}.
\]
It follows that $F\sim_{\mathcal A_{\infty}}G$, and hence,
$f\sim_{\mathcal A_0}g$.

Conversely, assume that $f\sim_{\mathcal A_0}g$. Then $f$ and $g$ have
representatives $\hat f$ and $\hat g$ which are topologically
equivalent to cones of maps of $S^1$ and there are 
homeomorphisms $\Sigma(\hat f)\approx \Sigma(\hat g)$, $\Delta(\hat
f)\approx \Delta(\hat g)$
and therefore also $\hat f^{-1}(\Delta(\hat f))\setminus \Sigma(\hat f) \approx
\hat g^{-1}(\Delta(\hat g))\setminus \Sigma(\hat g)$. By Lemma
\ref{L:1}, Corollary \ref{C:tran} and Lemma \ref{L:2}, when we pass to
the topologically equivalent cones of maps of circles, these sets
appear as disjoint curves in source and target intersecting each $t$-level
exactly once. It is clear that this implies that $[\as(\hat f|\tilde
S_{\epsilon}(\hat f))]_E=[\as(\hat g|\tilde S_{\epsilon}(\hat g))]_E$
and hence, that $\as(f)= \as(g)$.
\end{proof}

\subsection{Stable perturbations}\label{S:stablepert}
The notion of stable perturbations of generic smooth map-germs is
introduced in \cite{FukIsh} and is defined as follows: 
Let $f$ and $\hat f$ be as in Section \ref{S:geo} and let $\delta$ be
so small that both $ \hat f|\tilde E_{\delta}\setminus \{0\}:\tilde
E_{\delta}\setminus \{0\} \to
E_{\delta}\setminus \{0\}$ and $\hat f|\tilde S_{\delta}:\tilde
S_{\delta}\to S_{\delta}$ are $C^{\infty}$ stable. By Proposition
\ref{P:stab} such $\delta$ exist. Let $\tilde f: \tilde E_{\delta}\to
E_{\delta}$ be a stable map such that $\{p\in \tilde E_{\delta}\mid
\tilde f(p)\ne \hat f(p)\}\subset \mbox{int }\tilde E_{\delta}$. Such
a map $\tilde f$ is called a \emph{stable perturbation} of $f$. 

In \cite{FukIsh} it is shown that the number $\kappa(\tilde f)$ of
cusps of $\tilde f$ has to satisfy the formula
\[
\kappa(\tilde f) \equiv 1+ \frac 12\#\left\{\mbox{branches of }
  \Sigma(f)\setminus \{0\}\right\} + \deg f.
\]
Proposition \ref{P:deg}  enables us to reformulate this formula for
$\kappa(\tilde f)$ in terms of the components of $\as f$. 

\begin{proposition}
Let $f\in \mathcal O_{\mbox g}$ with $\as^{\#}(f)=[x_1,\ldots,x_n]_E$ and let $\tilde f$ be a stable
perturbation of $f$. Then the number $\kappa(\tilde f)$ of cusps of
$\tilde f$ satisfies
\[
\kappa(\tilde f)\equiv 1+\frac n2+\frac 1n\left|\sum_{i=1}^n(-1)^{i+1}[x_i+1]\right|
\mod 2.
\]
\end{proposition}
\begin{proof}
By Theorem 2.1 of \cite{FukIsh}, 
\[
\kappa(\tilde f)\equiv 1+ \frac 12\#\left\{\mbox{branches of }
  \Sigma(f)\setminus \{0\}\right\} + \deg f.
\]
By definition, $n= \#\left\{\mbox{branches of }
  \Sigma(f)\setminus \{0\}\right\}$ and furthermore, $|\deg \hat
f|\tilde S_{\epsilon}(\hat f)| = |\deg f|\equiv\deg f \mod 2$. By
Proposition \ref{P:deg}, $|\deg \hat
f|\tilde S_{\epsilon}(\hat f)| = \frac 1n\left|\sum_{i=1}^n(-1)^{i+1}[x_i+1]\right|$ and
this finishes the proof.
\end{proof} 

\subsection{Examples and tables}
When calculating $\as$, one has to check that the germ in question has
only fold singularities outside the origin. Let $\omega:\re 2\to \re 2$
be a smooth map and let $p\in \Sigma(\omega)$. It is shown in
\cite{BroSku}, Section 3, that $p$ is a fold point if
and only if 
\[
D\omega(p)\left(\begin{matrix}\phantom -\frac{\partial}{\partial
      y}J\omega(p)\\-\frac{\partial}{\partial
      x}J\omega(p)\end{matrix}\right)\ne \left(\begin{matrix} 0\\0\end{matrix}\right).
\]
For simplicity, put
\[
\nabla_{\perp}J\omega(p)=\left(\begin{matrix}\phantom -\frac{\partial}{\partial
      y}J\omega(p)\\-\frac{\partial}{\partial
      x}J\omega(p)\end{matrix}\right).
\]

\begin{example}\label{E:1}
Let $\omega(x,y)=(x,y^3+x^ky)$. We find $J\omega(x,y)=3y^2+x^k$, and
therefore, $\Sigma(\omega)$ is given by $3y^2+x^k=0$. We see that $k$
has to be odd in order for $\Sigma(\omega)\setminus \{0\}$ to be
non-empty. Assume that $k$ is odd. It is clear that $\omega^{-1}(0)=\{0\}$.

 The branches of $\omega$ is given
by 
\[
y=\pm \frac 1{\sqrt 3}(-x^k)^{\frac 12}.
\]
Let $z(x)=\frac 1{\sqrt 3}(-x^k)^{\frac 12}$. We compute 
\[
\omega(\pm z(x))= (x, \mp \frac 2{3\sqrt 3}(-x^k)^{\frac 32}).
\]
This shows that $\omega$ has no singular double points. Also, 
\[
D\omega(x,y)\nabla_{\perp}J\omega(x,y)=\left(\begin{matrix}6y\\0\end{matrix}\right)
\]
for $(x,y)\in \Sigma(\omega)$. This shows that $\omega$ has only fold
singularities outside the origin.

To find the branches of
$\omega^{-1}(\Delta(\omega))\setminus \Sigma(\omega)$, let $x<0$ and
consider $f_x(y)=y^3+x^ky$. We want to solve the equations 
\[
f_x(y)=f_x(z(x))
\]
 and 
\[f_x(y)=f_x(-z(x)).\]
Since $f_x$ is a polynomial of degree 3 in $y$ and $\pm z(x)$ are local
extremal points of $f_x$, there are $y_1(x)<-z(x)$ with $f_x(y_1(x))=f_x(z(x))$
and $y_2(x)>z(x)$ with $f_x(y_2(x))=f_x(-z(x))$. No other solutions
exist. We need to show that $y_1(x)\to 0$ as $x\to 0$ and $y_2(x)\to
0$ as $x\to 0$. We know that $f(x,\pm z(x))\to 0$ as $x\to
0$. Therefore, $f_x(y_1(x))=(y_1(x))^3+x^ky_1(x)\to 0$ as $x\to 0$ and
hence, $y_1(x)\to 0$ as $x\to 0$. The same argument applies to
$y_2$. Altogether, we have proved that 
\[
[\as(\omega)]_E=[(p,s,s,p)]_E.
\]
\end{example}

In \cite{Gaf}, T. Gaffney presents a table (\cite{Gaf}, 9.14) with
normal forms of topologically distinct map germs $\mathbb C^2 \to
\mathbb C^2$. Using theorem \ref{T:MT}, we are able to reduce this
list when we think of it as a list of map germs $\re 2 \to \re 2$.
Table \ref{Ta:G} is Gaffney's list of germs with the $\as$
calculated. We see that many of the $\mathcal A_0$-equivalence classes in Table
\ref{Ta:G} are the same in the real case. In the real case, Table
\ref{Ta:G}  reduces to Table \ref{Ta:S}.

\newpage

\begin{table}[!]
\begin{center}
\begin{tabular}{p{1in} p{2in} l}\hline
Type & & $[\as]_E$\\ \hline
(1) & $(x,y)$&$[(p)]_E$\\
(2) & $(x,y^2)$&$[(s,s)]_E$\\ 
(3) & $(x,xy+y^3)$&$[(p,s,s,p)]_E$\\
$\mbox{(4)}_k$ & $(x,y^3+x^ky)$& $[(p,s,s,p)]_E$\\
(5) & $(x,xy+y^4)$ & $[(s,s)]_E$\\
(6)&$(x,xy+y^5)$ & $[(p,s,s,p)]_E$\\
(7)&$(x,xy+y^6)$ & $[(s,s)]_E$\\
(8)&$(x,xy+y^7)$ & $[(p,s,s,p)]_E$\\
$\mbox{(9)}_{2k+1}$&$(x,xy^2+y^4+y^{2k+1})\quad$ & $[(p,s,s,p,p,s,s,p)]_E$\\
(10)&$(x,xy^2+y^5)$&$[(p,s,s,p,p,s,s,p)]_E$\\
(11)&$(x,x^2y+y^4)$&$[(s,s)]_E$\\
(12)&$(x,xy^2+y^6+y^7)$&$[(s,p,s,s,p,s,p,p)]_E$\\
(13)&$(x,x^2y+xy^3+y^5)$&$[(p)]_E$\\ 
(14)&$(x,x^3y+y^4+x^3y^2)$&$[(s,s)]_E$\\ \hline
\end{tabular}
\end{center} 
\caption{Gaffney's table}\label{Ta:G}
\end{table} 

\begin{table}[!]
\begin{center}
\begin{tabular}{p{1in} p{2in} l}\hline
Type & & $\as$\\ \hline
(1) & $(x,y)$&$[(p)]_E$\\
(2) & $(x,y^2)$&$[(s,s)]_E$\\ 
(3) & $(x,xy+y^3)$&$[(p,s,s,p)]_E$\\
(4)&$(x,xy^2+y^5)$&$[(p,s,s,p,p,s,s,p)]_E$\\
(5)&$(x,xy^2+y^6+y^7)$&$[(s,p,s,s,p,s,p,p)]_E$\\\hline
\end{tabular}
\end{center} 
\caption{Reduced table}\label{Ta:S}
\end{table} 

\subsection{Acknowledgements} The author wishes to thank Hans
Brodersen for ideas and corrections, and Magnus D. Vigeland for
suggesting the proof of Proposition \ref{P:mn4}.

\bibliography{ref}

\end{document}